\documentclass[12pt,twoside,reqno,psamsfonts]{amsart}

\usepackage[OT1]{fontenc}
\usepackage{type1cm}
\usepackage{amssymb}
\usepackage[dvips]{graphicx}
\usepackage{geometry}        % for a good layout
\usepackage{version}         % include, exclude and comment

\numberwithin{equation}{section}

\theoremstyle{plain}
\newtheorem{theorem}{Theorem}[section]

\theoremstyle{remark}
\newtheorem{remark}[theorem]{Remark}

\theoremstyle{definition}

\newtheorem{question}[theorem]{Question}

\newcommand{\HH}{\mathcal{H}}
\newcommand{\PP}{\mathcal{P}}

\newcommand{\R}{\mathbb{R}}

\newcommand{\N}{\mathbb{N}}

\renewcommand{\atop}[2]{\genfrac{}{}{0pt}{}{#1}{#2}}

\DeclareMathOperator{\dimh}{dim_H}
\DeclareMathOperator{\dimp}{dim_P}
\DeclareMathOperator{\dist}{dist}

\DeclareMathOperator{\proj}{proj}

\begin{document}

\title{On upper conical density results}

\author{Antti K\"aenm\"aki}
\address{Department of Mathematics and Statistics \\
         P.O. Box 35 (MaD) \\
         FI-40014 University of Jyv\"askyl\"a \\
         Finland}
\email{antti.kaenmaki@jyu.fi}

%\thanks{}
\date{\today}

\begin{abstract}
  We report a recent developement on the theory of upper conical densities. More precicely, we look at what can be said in this respect for other measures than just the Hausdorff measure. We illustrate the methods involved by proving a result for the packing measure and for a purely unrectifiable doubling measure.
\end{abstract}

\maketitle

\section{Introduction}

Conical density theorems are used in geometric measure theory to derive geometric information from given metric information. Classically, they deal with the distribution of the Hausdorff measure. The main applications of upper conical density results concern rectifiability and porosity. The extensive study of upper conical densities was pioneered by Besicovitch \cite{Besicovitch1938} who studied the conical density properties of purely $1$-unrectifiable sets in the plane. Besides Besicovitch, the theory of upper conical densities has been developed by Morse and Randolph \cite{MorseRandolph1944}, Marstrand \cite{Marstrand1954}, and Federer \cite{Federer1969}.

\section{Notation and preliminaries}

Let $n \in \N$, $m \in \{ 0,\ldots,n-1 \}$, and $G(n,n-m)$ denote the space of all $(n-m)$-dimensional linear subspaces of $\R^n$. The unit sphere of $\R^n$ is denoted by $S^{n-1}$. For $x \in \R^n$, $\theta \in S^{n-1}$, $0 < \alpha \le 1$, $r>0$, and $V \in G(n,n-m)$, we set
\begin{align*}
  H(x,\theta,\alpha) &= \{ y \in \R^n : (y-x) \cdot \theta > \alpha
  |y-x| \}, \\
  X(x,V,\alpha) &= \{ y \in \R^n : \dist(y-x,V) < \alpha|y-x| \}, \\
  X(x,r,V,\alpha) &= B(x,r) \cap X(x,V,\alpha),
\end{align*}
where $B(x,r)$ is the closed ball centered at $x$ with radius $r$. We also denote $B_V(x,r) = \proj_V\bigl( B(x,r) \bigr)$, where $\proj_V$ is the orthogonal projection onto $V$.

By a measure we will always mean a locally finite nontrivial Borel regular (outer) measure defined on all subsets of $\R^n$. We use the notation $\HH^s$ and $\PP^s$ to denote the $s$-dimensional Hausdorff and packing measure, respectively. Consult \cite[\S 4 and \S 5.10]{Mattila1995}. We follow the convention according to which $c=c(\cdots)>0$ denotes a constant depending only on the parameters listed inside the parentheses.

If $A \subset \R^n$ is a Borel set with $0<\HH^s(A)<\infty$, then \cite[Theorem 6.2(1)]{Mattila1995} implies that for $\HH^s$-almost every $x \in A$ there are arbitrary small radii $r$ so that $\HH^s\bigl( A \cap B(x,r) \bigr)$ is proportional to $r^s$. So we know roughly how much of $A$ there is in such small balls $B(x,r)$. But we would also like to know how the set $A$ is distributed there. The following three upper conical density theorems give information how much there is $A$ near $(n-m)$-planes in the sense of the Hausdorff measure.

\begin{theorem}[\mbox{\cite[Theorem 3.1]{Salli1985}}] \label{thm:salli}
  If $n \in \N$, $m \in \{ 0,\ldots,n-1 \}$, $m<s\le n$, and $0<\alpha\le 1$, then there is a constant $c=c(n,m,s,\alpha)>0$ satisfying the following: For every $A \subset \R^n$ with $\HH^s(A)<\infty$ and for each $V \in G(n,n-m)$ it holds that
  \begin{equation*}
    \limsup_{r \downarrow 0} \frac{\HH^s\bigl( A \cap X(x,r,V,\alpha) \bigr)}{(2r)^s} \ge c
  \end{equation*}
  for $\HH^s$-almost every $x \in A$.
\end{theorem}

\begin{remark} \label{rem:rect}
  The role of the assumption $s>m$ in the above theorem is to guarantee that the set $A$ is scattered enough. When $s \le m$, it might happen that $A \in V^\bot$ for some $V \in G(n,n-m)$. In fact, if the set $A$ is $m$-rectifiable with $\HH^m(A)<\infty$, then it follows from \cite[Theorem 15.19]{Mattila1995} that the result of Theorem \ref{thm:salli} cannot hold. On the other hand, if the set $A$ is purely $m$-unrectifiable with $\HH^m(A)<\infty$, then the result of Theorem \ref{thm:salli} holds, see \cite[Corollary 15.16]{Mattila1995}. We refer the reader to \cite[\S 15]{Mattila1995} for the basic properties of rectifiable and purely unrectifiable sets. See also \S \ref{sec:unrect}.
\end{remark}

\begin{theorem}[\mbox{\cite[Theorem 3.3]{Mattila1988}}] \label{thm:mattila}
  If $n \in \N$, $m \in \{ 0,\ldots,n-1 \}$, $m<s\le n$, and $0<\alpha\le 1$, then there is a constant $c=c(n,m,s,\alpha)>0$ satisfying the following: For every $A \subset \R^n$ with $\HH^s(A)<\infty$ it holds that
  \begin{equation*}
    \limsup_{r \downarrow 0} \inf_{V \in G(n,n-m)} \frac{\HH^s\bigl( A \cap X(x,r,V,\alpha) \bigr)}{(2r)^s} \ge c
  \end{equation*}
  for $\HH^s$-almost every $x \in A$.
\end{theorem}

The above theorem is a significant improvement of Theorem \ref{thm:salli}. It shows that in the sense of the Hausdorff measure, there are arbitrary small scales so that almost all points of $A$ are well surrounded by $A$. Theorem \ref{thm:mattila} is actually applicable for more general symmetric cones. More precisely, the set $X(x,r,V,\alpha)$ can be replaced by $C_x \cap B(x,r)$, where $C_x=\bigcup_{V \in C}(V+x)$, $C \subset G(n,n-m)$ is a Borel set with $\gamma(C)>\delta>0$, and $\gamma$ is the natural isometry invariant measure on $G(n,n-m)$. The infimum is then taken over all such sets $C$. The proof of Theorem \ref{thm:mattila} (and its more general formulation) is nontrivial and it is based on Fubini-type arguments and an elegant use of the so-called sliced measures. The following theorem gives Theorem \ref{thm:mattila} a more elementary proof. The technique used there does not require the cones to be symmetric.

\begin{theorem}[\mbox{\cite[Theorem 2.5]{KaenmakiSuomala2004}}] \label{thm:hausdorff}
  If $n \in \N$, $m \in \{ 0,\ldots,n-1 \}$, $m<s\le n$, and $0<\alpha\le 1$, then there is a constant $c=c(n,m,s,\alpha)>0$ satisfying the following: For every $A \subset \R^n$ with $\HH^s(A)<\infty$ it holds that
  \begin{equation*}
    \limsup_{r \downarrow 0} \inf_{\atop{\theta \in S^{n-1}}{V \in G(n,n-m)}} \frac{\HH^s\bigl( A \cap X(x,r,V,\alpha) \setminus H(x,\theta,\alpha) \bigr)}{(2r)^s} \ge c
  \end{equation*}
  for $\HH^s$-almost every $x \in A$.
\end{theorem}

The main application of this theorem is porosity. With porous sets we mean sets which have holes on every small scale. For precise definition of the porous set and the connection between porosity and upper conical densities, the reader is referred to \cite[Theorem 11.14]{Mattila1995}, \cite[Theorem 3.2]{KaenmakiSuomala2004}, and \cite[\S 3]{KaenmakiSuomala2008}. See \cite{EckmannJarvenpaaJarvenpaa2000, BeliaevSmirnov2002, JarvenpaaJarvenpaa2002, JarvenpaaJarvenpaaKaenmakiSuomala2005, RajalaSmirnov2007, JarvenpaaJarvenpaaKaenmakiRajalaRogovinSuomala2007, Chousionis2008b, Rajala2009} for other related results.

We will now look at what kind of upper conical density results can be proven for other measures.

\section{Packing type measures}

The following result is Theorem \ref{thm:hausdorff} for the packing measure. A more general formulation can be found in \cite{KaenmakiSuomala2008}. To our knowledge, it is the first upper conical density result for other measures than the Hausdorff measure.

\begin{theorem} \label{thm:packing}
  If $n \in \N$, $m \in \{ 0,\ldots,n-1 \}$, $m<s\le n$, and $0<\alpha\le 1$, then there is a constant $c=c(n,m,s,\alpha)>0$ satisfying the following: For every $A \subset \R^n$ with $\PP^s(A)<\infty$ it holds that
  \begin{equation*}
    \limsup_{r \downarrow 0} \inf_{\atop{\theta \in S^{n-1}}{V \in G(n,n-m)}} \frac{\PP^s\bigl( A \cap X(x,r,V,\alpha) \setminus H(x,\theta,\alpha) \bigr)}{(2r)^s} \ge c
  \end{equation*}
  for $\PP^s$-almost every $x \in A$.
\end{theorem}

\begin{proof}
  Fix $n \in \N$, $m \in \{ 0,\ldots,n-1 \}$, and $0<\alpha\le 1$.
  Observe that $G(n,n-m)$ endowed with the metric $d(V,W) = \sup_{x \in V \cap S^{n-1}} \dist(x,W)$ is a compact metric space and
  \begin{equation}
    \bigcup_{d(V,W)<\alpha} \{ x : x \in W \} = X(0,V,\alpha)
  \end{equation}
  for all $V \in G(n,n-m)$, see \cite[Lemma 2.2]{Salli1985}. Using the compactness, we may thus choose $K=K(n,m,\alpha) \in \N$ and $(n-m)$-planes $V_1,\ldots,V_K$ so that for each $V \in G(n,n-m)$ there is $j \in \{ 1,\ldots,K \}$ with
  \begin{equation*}
    X(x,V,\alpha) \supset X(x,V_j,\alpha/2).
  \end{equation*}

  Let $t = \max\{ t(\alpha/2),1+6/\alpha \} \ge 1$, where $t(\alpha/2)$ is as in \cite[Lemma 4.3]{CsornyeiKaenmakiRajalaSuomala2009} and take $q=q(n,\alpha/(2t))$ from \cite[Lemma 4.2]{CsornyeiKaenmakiRajalaSuomala2009}. If $\lambda > 0$ and $\{ B(y_i,t\lambda) \}_{i=1}^q$ is a collection of pairwise disjoint balls centered at $B(0,1)$, then it follows from the above mentioned lemmas that there exists $i_0 \in \{ 1,\ldots,q \}$ such that for every $\theta \in S^{n-1}$ there is $k \in \{ 1,\ldots,q \}$ for which
  \begin{equation} \label{eq:erdos}
    B(y_k,\lambda) \subset B(y_{i_0},3) \setminus H(y_{i_0},\theta,\alpha).
  \end{equation}
  In fact, there are three center points that form a large angle. The choice of $t$ also implies that for every $y,y_0 \in \proj_{V^\bot}^{-1}\bigl( B_{V^\bot}(0,\lambda) \bigr)$ with $|y-y_0| \ge t\lambda$, we have
  \begin{equation} \label{eq:hippatsuikka}
    B(y,\lambda) \subset X(y_0,V,\alpha/2).
  \end{equation}
  Let $c_1=c_1(m)$ and $c_2=c_2(n)$ be such that the set $B_{V^\bot}(0,1)$ may be covered by $c_1\lambda^{-m}$ and the set $\proj_{V^\bot}^{-1}\bigl( B_{V^\bot}(0,\lambda) \bigr) \cap B(0,1)$ may be covered by $c_2\lambda^{m-n}$ balls of radius $\lambda$ for all $V \in G(n,n-m)$.   Finally, fix $m<s\le n$ and set $\lambda=\lambda(n,m,s,\alpha) = \min\{ 2^{-1}t^{s/(m-s)}d^{1/(s-m)},(3t)^{-1} \} > 0$, where $d=d(n,m,\alpha)=1/(2c_1K(q-1))$.

  It suffices to show that if $c = \lambda^n/(3^s4c_1c_2K) > 0$ and $A \subset \R^n$ with $\PP^s(A)<\infty$, then
  \begin{equation*} %\label{eq:new_claim}
    \limsup_{r \downarrow 0} \inf_{\atop{\theta \in S^{n-1}}{j \in \{ 1,\ldots,K \}}} \frac{\PP^s\bigl( A \cap X(x,r,V_j,\alpha/2) \setminus H(x,\theta,\alpha) \bigr)}{(2r)^s} \ge c
  \end{equation*}
  for $\PP^s$-almost every $x \in A$. Assume to the contrary that there are $r_0>0$ and a closed set $A \subset \R^n$ with $0<\PP^s(A)<\infty$ so that for every $x \in A$ and $0<r<r_0$ there exist $j \in \{ 1,\ldots,K \}$ and $\theta \in S^{n-1}$ such that
  \begin{equation*} %\label{eq:P_antithesis}
    \PP^s\bigl( A \cap X(x,r,V_j,\alpha/2) \setminus H(x,\theta,\alpha) \bigr) < c(2r)^s.
  \end{equation*}
  Recalling \cite[Theorem 6.10]{Mattila1995}, we may further assume that
  \begin{equation} \label{eq:taylor_tricot}
    \liminf_{r \downarrow 0} \frac{\PP^s\bigl( A \cap B(x,r) \bigr)}{(2r)^s} = 1
  \end{equation}
  for all $x \in A$. Pick $x_0 \in A$ and choose $0<r'<r_0$ so that $\PP^s\bigl( A \cap B(x_0,r') \bigr) \ge (2r')^s/2$. For notational simplicity, we assume that $r'=1$ and $r_0>3$. Since $A = \bigcup_{j=1}^K A_j$, where
  \begin{equation} \label{eq:j_antithesis}
    A_j = \{ x \in A : \PP^s\bigl( A \cap X(x,3,V_j,\alpha/2) \setminus H(x,\theta,\alpha) \bigr) < c6^s \text{ for some } \theta \},
  \end{equation}
  we find $j \in \{ 1,\ldots,K \}$ for which $\PP^s\bigl( A_j \cap B(x_0,1) \bigr) \ge 2^s/(2K)$. Going into a subset, if necessary, we may assume that $A_j \cap B(x_0,1)$ is compact. Moreover, we may cover the set $B_{V_j^\bot}(x_0,1)$ by $c_1\lambda^{-m}$ balls of radius $\lambda$. Hence
  \begin{equation} \label{eq:slice}
    \PP^s\bigl( A_j \cap \proj_{V_j^\bot}^{-1}\bigl( B_{V_j^\bot}(y',\lambda) \bigr) \cap B(x_0,1) \bigr) \ge \lambda^m 2^s/(2c_1K)
  \end{equation}
  for some $y' \in V_j^\bot$. Next we choose $q$ pairwise disjoint balls $\{ B(y_i,t\lambda) \}_{i=1}^q$ centered at $A_j' = A_j \cap \proj_{V_j^\bot}^{-1}\bigl( B_{V_j^\bot}(y',\lambda) \bigr) \cap B(x_0,1)$ so that for each $i \in \{ 1,\ldots,q \}$ it holds $\PP^s\bigl( A \cap B(y_i,\lambda) \bigr) \ge \PP^s\bigl( A \cap B(y,\lambda) \bigr)$ for all $y \in A_j' \setminus \bigcup_{k=1}^{i-1} U(y_k,t\lambda)$, where $U(x,r)$ denotes the open ball. This can be done since the set $A_j'$ is compact and the function $y \mapsto \PP^s\bigl( A \cap B(y,\lambda) \bigr)$ is upper semicontinuous. The set $A_j'$ can be covered by $c_2\lambda^{m-n}$ balls of radius $\lambda$, whence
  \begin{equation} \label{eq:q_arvio}
    c_2\lambda^{m-n} \PP^s\bigl( A \cap B(y_q,\lambda) \bigr) \ge \lambda^m2^s/(2c_1K) - \sum_{i=1}^{q-1} \PP^s\bigl( A \cap B(y_i,t\lambda) \bigr)
  \end{equation}
  by recalling \eqref{eq:slice}. Now \eqref{eq:erdos}, \eqref{eq:hippatsuikka}, and \eqref{eq:j_antithesis} give
  \begin{align*}
    \PP^s\bigl( A \cap B(y_q,\lambda) \bigr) \le \PP^s\bigl( A \cap X(y_{i_0},3,V_j,\alpha/2) \setminus H(y_{i_0},\theta,\alpha) \bigr) < c6^s,
  \end{align*}
  and consequently,
  \begin{equation*}
    \sum_{i=1}^{q-1} \PP^s\bigl( A \cap B(y_i,t\lambda) \bigr) \ge \lambda^m2^s/(2c_1K) - c6^sc_2\lambda^{m-n} = (q-1)2^sd\lambda^m/2
  \end{equation*}
  by \eqref{eq:q_arvio} and the choices of $c$ and $d$. Hence $\PP^s\bigl( A \cap B(x_1,t\lambda) \bigr) \ge 2^sd\lambda^m/2$ for some $x_1 \in \{ y_1,\ldots,y_{q-1} \} \subset A \cap B\bigl( x_0,1 \bigr)$. Recall that $\PP^s\bigl( A \cap B(x_0,1) \bigr) \ge 2^s/2$. Repeating now the above argument in the ball $B(x_1,t\lambda)$, we find a point $x_2 \in A \cap B(x_1,t\lambda)$ so that $\PP^s\bigl( A \cap B(x_2,(t\lambda)^2) \bigr) \ge 2^sd^2\lambda^{2m}/2$. Continuing in this manner, we find for each $k \in \N$ a ball $B(x_k,(t\lambda)^k)$ centered at $A \cap B(x_{k-1},(t\lambda)^{k-1})$ so that $\PP^s\bigl( A \cap B(x_k,(t\lambda)^k) \bigr) \ge 2^sd^k\lambda^{km}/2$.

  Now for the point $z \in A$ determined by $\{ z \} = \bigcap_{k=0}^\infty B\bigl( x_k,(t\lambda)^k \bigr)$, we have
  \begin{align*}
    \liminf_{r \downarrow 0} \frac{\PP^s\bigl( A \cap B(z,r) \bigr)}{(2r)^s} &\ge \liminf_{k \to \infty} \frac{\PP^s\bigl( A \cap B(x_{k+1},(t\lambda)^{k+1}) \bigr)}{2^s(t\lambda)^{(k-1)s}} \\ &\ge \liminf_{k \to \infty} \frac{d^{k+1}\lambda^{(k+1)m}}{2^s(t\lambda)^{(k-1)s}} \\ &= \liminf_{k \to \infty} 2^{-s}d^2\lambda^{2m}(d\lambda^{m-s}t^{-s})^{k-1} \\ &\ge \liminf_{k \to \infty} 2^{-s}d^2\lambda^{2m}2^{(s-m)(k-1)} = \infty,
  \end{align*}
  since $t\lambda \le 1/3$, $s>m$ and $\lambda^{m-s} \ge 2^{s-m}t^sd^{-1}$. This contradicts with \eqref{eq:taylor_tricot}. The proof is finished.
\end{proof}

The above result is a special case of the following more general result.

\begin{theorem}[\mbox{\cite[Theorem 2.4]{KaenmakiSuomala2008}}] \label{thm:adv_result}
  If $n \in \N$, $m \in \{ 0,\ldots,n-1 \}$, $0<\alpha\le 1$, and a nondecreasing function $h \colon (0,\infty) \to (0,\infty)$ satisfies
  \begin{equation} \label{eq:gauge}
    \limsup_{r \downarrow 0} \frac{h(\gamma r)}{h(r)} < \gamma^m
  \end{equation}
  for some $0<\gamma<1$, then there is a constant $c=c(n,m,h,\alpha)>0$ satisfying the following: For every measure $\mu$ on $\R^n$ with
  \begin{equation*}
    \liminf_{r \downarrow 0} \frac{\mu\bigl( B(x,r) \bigr)}{h(2r)} < \infty \quad \text{for $\mu$-almost all $x \in \R^n$}
  \end{equation*}
  it holds that
  \begin{equation*}
    \limsup_{r \downarrow 0} \inf_{\atop{\theta \in S^{n-1}}{V \in G(n,n-m)}} \frac{\mu\bigl( X(x,r,V,\alpha) \setminus H(x,\theta,\alpha) \bigr)}{h(2r)} \ge c \limsup_{r \downarrow 0} \frac{\mu\bigl( B(x,r) \bigr)}{h(2r)}
  \end{equation*}
  for $\mu$-almost every $x \in \R^n$.
\end{theorem}

\begin{remark}
  If instead of \eqref{eq:gauge}, the function $h \colon (0,\infty) \to (0,\infty)$ satisfies
  \begin{equation*}
    \liminf_{r \downarrow 0} \frac{h(\gamma r)}{h(r)} \ge \gamma^m
  \end{equation*}
  for all $0<\gamma<1$, then \cite[Proposition 3.3]{KaenmakiSuomala2008} implies that the result of Theorem \ref{thm:adv_result} cannot hold. This shows that Theorem \ref{thm:adv_result} fails for gauge functions such as $h(r) = r^m/\log(1/r)$ when $m > 0$.
\end{remark}

\section{Measures with positive dimension}

When working with a Hausdorff or packing type measure $\mu$, it is useful to
study densities such as
\begin{equation*}
  \limsup_{r \downarrow 0} \frac{\mu\bigl(X(x,r,V,\alpha)\bigr)}{h(2r)},
\end{equation*}
where $h$ is the gauge function used to construct the measure $\mu$. However, most measures are so unevenly distributed that there are no gauge functions that could be used to approximate the measure in small balls. To
obtain conical density results for general measures it seems natural to replace the value of the gauge $h$ in the denominator by the measure of the ball $B(x,r)$.

The following result is valid for all measures on $\R^n$.

\begin{theorem}[\mbox{\cite[Theorem 3.1]{CsornyeiKaenmakiRajalaSuomala2009}}] \label{thm:all}
  If $n \in \N$ and $0 < \alpha \le 1$, then there is a constant $c = c(n,\alpha)>0$ satisfying the following: For every measure $\mu$ on $\R^n$ it holds that
  \begin{equation*}
    \limsup_{r \downarrow 0} \inf_{\theta \in S^{n-1}} \frac{\mu\bigl(
      B(x,r) \setminus H(x,\theta,\alpha) \bigr)}{\mu\bigl( B(x,r)
      \bigr)} \ge c
  \end{equation*}
  for $\mu$-almost every $x \in \R^n$.
\end{theorem}

By assuming a lower bound for the Hausdorff dimension of the measure, the measure will be scattered enough so that we are able to prove a result similar to Theorem \ref{thm:hausdorff} for general measures. The \emph{(lower) Hausdorff and packing dimensions of a measure $\mu$} are defined by
\begin{align*}
  \dimh(\mu) &= \inf\{ \dimh(A) : A \text{ is a Borel set with }
  \mu(A)>0 \}, \\
  \dimp(\mu) &= \inf\{ \dimp(A) : A \text{ is a Borel set with }
  \mu(A)>0 \},
\end{align*}
where $\dimh(A)$ and $\dimp(A)$ denote the Hausdorff and packing dimensions of the set $A \subset \R^n$, respectively. The reader is referred to \cite[\S 4 and \S 5.9]{Mattila1995}.

\begin{theorem}[\mbox{\cite[Theorem 4.1]{CsornyeiKaenmakiRajalaSuomala2009}}] \label{thm:dimh_positive}
  If $n \in \N$, $m \in \{ 0,\ldots,n-1 \}$, $m<s\le n$, and $0<\alpha\le 1$, then there is a constant $c=c(n,m,s,\alpha)>0$ satisfying the following: For every measure $\mu$ on $\R^n$ with $\dimh(\mu) \ge s$ it holds that
  \begin{equation*} %\label{eq:claim}
    \limsup_{r \downarrow 0} \inf_{\atop{\theta \in S^{n-1}}{V \in
        G(n,n-m)}} \frac{\mu\bigl( X(x,r,V,\alpha) \setminus
      H(x,\theta,\alpha) \bigr)}{\mu\bigl( B(x,r) \bigr)} \ge c
  \end{equation*}
  for $\mu$-almost every $x \in \R^n$.
\end{theorem}

\begin{question}
  Does Theorem \ref{thm:dimh_positive} hold if we just assume $\dimp(\mu) \ge s$ instead of $\dimh(\mu) \ge s$?
\end{question}

\section{Purely unrectifiable measures} \label{sec:unrect}

Another condition to guarantee the measure to be scattered enough is unrectifiability. A measure on $\R^n$ is called \emph{purely $m$-unrectifiable} if $\mu(A)=0$ for all $m$-rectifiable sets $A \subset \R^n$. We refer the reader to \cite[\S 15]{Mattila1995} for the basic properties of rectifiable sets. Applying the ideas of \cite[Lemma 15.14]{Mattila1995}, we are able to prove the following theorem.

\begin{theorem} \label{thm:unrect}
  If $d > 0$ and $0<\alpha\le 1$, then there is a constant $c=c(d,\alpha)>0$ satisfying the following: For every $n \in \N$, $m \in \{ 1,\ldots,n-1 \}$, $V \in G(n,n-m)$, and purely $m$-unrectifiable measure $\mu$ on $\R^n$ with
  \begin{equation} \label{eq:doubling}
    \limsup_{r \downarrow 0} \frac{\mu\bigl( B(x,2r) \bigr)}{\mu\bigl( B(x,r) \bigr)} < d \quad \text{for $\mu$-almost all $x \in \R^n$}
  \end{equation}
  it holds that
  \begin{equation} \label{eq:unrect_claim}
    \limsup_{r \downarrow 0} \frac{\mu\bigl( X(x,r,V,\alpha) \bigr)}{\mu\bigl( B(x,r) \bigr)} \ge c
  \end{equation}
  for $\mu$-almost every $x \in \R^n$.
\end{theorem}

\begin{proof}
  Fix $d>0$ and $0<\alpha\le 1$. Observe that there exists a constant $b=b(d,\alpha)>0$ such that any measure $\mu$ satisfying \eqref{eq:doubling} fulfills
  \begin{equation*}
    \limsup_{r \downarrow 0} \frac{\mu\bigl( B(x,3r) \bigr)}{\mu\bigl( B(x,\alpha r/20) \bigr)} < b
  \end{equation*}
  for $\mu$-almost all $x \in \R^n$.

  We will prove that \eqref{eq:unrect_claim} holds with $c=(4bd)^{-1}$. Assume to the contrary that for some $n \in \N$, $m \in \{ 1,\ldots,n-1 \}$, $V \in G(n,n-m)$, and purely $m$-unrectifiable measure $\mu$ satisfying \eqref{eq:doubling} there exist $r_0>0$ and a Borel set $A \subset \R^n$ with $\mu(A)>0$ so that
  \begin{equation} \label{eq:yksi}
    \mu\bigl( X(x,2r,V,\alpha) \bigr) < c\mu\bigl( B(x,2r) \bigr)
  \end{equation}
  for all $0<r<r_0$ and for every $x \in A$. We may further assume that
  \begin{align}
    \mu\bigl( B(x,2r) \bigr) &\le d\mu\bigl( B(x,r) \bigr), \label{eq:kolme} \\
    \mu\bigl( B(x,3r) \bigr) &\le b\mu\bigl( B(x,\alpha r/20) \bigr) \label{eq:kaksi}
  \end{align}
  for all $0<r<r_0$ and for every $x \in A$.

  Recalling \cite[Corollary 2.14(1)]{Mattila1995}, we fix $x_0 \in A$ and $0<r<r_0$ so that
  \begin{equation} \label{eq:nelja}
    \mu\bigl( A \cap B(x_0,r) \bigr) > \mu\bigl( B(x_0,r) \bigr)/2.
  \end{equation}
  For each $x \in A \cap B(x_0,r)$ we define $h(x) = \sup\{ |y-x| : y \in A \cap X(x,r,V,\alpha/4) \}$. Since $\mu$ is purely $m$-unrectifiable, it follows from \cite[Lemma 15.13]{Mattila1995} that $h(x)>0$ for $\mu$-almost all $x \in A \cap B(x_0,r)$. For each $x \in A \cap B(x_0,r)$ with $h(x)>0$ we choose $y_x \in A \cap X(x,r,V,\alpha/4)$ such that $|y_x-x|>3h(x)/4$. Inspecting the proof of \cite[Lemma 15.14]{Mattila1995}, we see that
  \begin{equation} \label{eq:viisi}
    A \cap \proj_{V^\bot}^{-1}\bigl( B_{V^\bot}(x,\alpha h(x)/4) \bigr) \subset X(x,2h(x),V,\alpha) \cup X(y_x,2h(x),V,\alpha)
  \end{equation}
  for all $x \in A \cap B(x_0,r)$. Applying the $5r$-covering theorem (\cite[Theorem 2.1]{Mattila1995}) to the collection $\bigl\{ B_{V^\bot}\bigl( x,\alpha h(x)/4 \bigr) : x \in A \cap B(x_0,r) \textrm{ with } h(x)>0 \bigr\}$, we find a countable collection of pairwise disjoint balls $\bigl\{ B_{V^\bot}\bigl(x_i,\alpha h(x_i)/20\bigr) \bigr\}_i$ so that
  \begin{equation} \label{eq:seitseman}
    \bigcup_{h(x)>0} \proj_{V^\bot}^{-1}\bigl( B_{V^\bot}(x,\alpha h(x)/4) \bigr) \subset \bigcup_i \proj_{V^\bot}^{-1}\bigl( B_{V^\bot}(x_i,\alpha h(x_i)/4) \bigr).
  \end{equation}
  Now \eqref{eq:seitseman}, \eqref{eq:viisi}, \eqref{eq:yksi}, \eqref{eq:kaksi}, and \eqref{eq:kolme} imply
  \begin{align*}
    \mu\bigl( A \cap B(x_0,r) \bigr) &\le \sum_i \mu\bigl( A \cap B(x_0,r) \cap \proj_{V^\bot}^{-1}\bigl( B_{V^\bot}( x_i,\alpha h(x_i)/4) \bigr) \bigr) \\
    &\le c\sum_i \mu\bigl( B(x_i,2h(x_i)) \bigr) + c\sum_i \mu\bigl( B(y_{x_i},2h(x_i)) \bigr) \\
    &\le 2c\sum_i \mu\bigl( B(x_i,3h(x_i)) \bigr) \le 2cb\sum_i \mu\bigl( B(x_i,\alpha h(x_i)/20) \bigr) \\
    &\le 2cb \mu\bigl( B(x_0,2r) \bigr) \le 2cbd \mu\bigl( B(x_0,r) \bigr) = \mu\bigl( B(x_0,r) \bigr)/2,
  \end{align*}
  that is, a contradiction with \eqref{eq:nelja}. The proof is finished.
\end{proof}

\begin{remark}
  Theorem \ref{thm:unrect} does not hold without the assumption \eqref{eq:doubling}, see \cite[Example 5.5]{CsornyeiKaenmakiRajalaSuomala2009} for a counterexample. Recall also Remark \ref{rem:rect}. Observe that one cannot hope to generalize the result by taking the infimum over all $V \in G(n,n-m)$ before taking the $\limsup$ as in Theorem \ref{thm:hausdorff}. A counterexample follows immediately from \cite[Example 5.4]{CsornyeiKaenmakiRajalaSuomala2009} by noting that the set constructed in the example supports a $1$-regular measure, that is, a measure giving for each small ball measure proportional to the radius. See also \cite[Proposition 3.3 and Remark 3.4]{KaenmakiSuomala2008} and \cite[Question 4.2]{KaenmakiSuomala2008} for related discussion.
\end{remark}

%
%
% BibTeX users please use
% \bibliographystyle{abbrv} %or alpha
% \bibliography{Bibliography}

\begin{thebibliography}{10}

\bibitem{Besicovitch1938}
A.~S.~Besicovitch.
\newblock On the fundamental geometrical properties of linearly measurable
  plane sets of points {II}.
\newblock {\em Math. Ann.}, 115:296--329, 1938.

\bibitem{RajalaSmirnov2007}
D.~Beliaev, E.~J{\"a}rvenp{\"a}{\"a}, M.~J{\"a}rvenp{\"a}{\"a}, A.~K{\"a}enm{\"a}ki, T.~Rajala, S.~Smirnov, and V.~Suomala.
\newblock Packing dimension of mean porous measures.
\newblock 2007.
\newblock University of Jyv{\"a}skyl{\"a}, preprint 341,
  http://www.jyu.fi/science/laitokset/maths/tutkimus/preprints.

\bibitem{BeliaevSmirnov2002}
D.~Beliaev and S.~Smirnov.
\newblock On dimension of porous measures.
\newblock {\em Math. Ann.}, 323(1):123--141, 2002.

\bibitem{Chousionis2008b}
V.~Chousionis.
\newblock Directed porosity on conformal iterated function systems and weak
  convergence of singular integrals.
\newblock {\em Ann. Acad. Sci. Fenn. Math.}, 2009.
\newblock to appear.

\bibitem{CsornyeiKaenmakiRajalaSuomala2009}
M.~Cs{\"o}rnyei, A.~K{\"a}enm{\"a}ki, T.~Rajala, and V.~Suomala.
\newblock Upper conical density results for general measures on {$\R^n$}.
\newblock {\em Proc. Edinburgh Math. Soc.}, 2009.
\newblock to appear.

\bibitem{EckmannJarvenpaaJarvenpaa2000}
J.-P.~Eckmann, E.~J{\"a}rvenp{\"a}{\"a}, and M.~J{\"a}rvenp{\"a}{\"a}.
\newblock Porosities and dimensions of measures.
\newblock {\em Nonlinearity}, 13(1):1--18, 2000.

\bibitem{Federer1969}
H.~Federer.
\newblock {\em Geometric Measure Theory}.
\newblock Springer-Verlag, Berlin, 1969.

\bibitem{JarvenpaaJarvenpaa2002}
E.~J{\"a}rvenp{\"a}{\"a} and M.~J{\"a}rvenp{\"a}{\"a}.
\newblock Porous measures on {$\R^n$}: local structure and dimensional
  properties.
\newblock {\em Proc. Amer. Math. Soc.}, 130(2):419--426, 2002.

\bibitem{JarvenpaaJarvenpaaKaenmakiRajalaRogovinSuomala20%
07}
E.~J{\"a}rvenp{\"a}{\"a}, M.~J{\"a}rvenp{\"a}{\"a}, A.~K{\"a}enm{\"a}ki, T.~Rajala, S.~Rogovin, and V.~Suomala.
\newblock Packing dimension and {A}hlfors regularity of porous sets in metric
  spaces.
\newblock 2007.
\newblock University of Jyv{\"a}skyl{\"a}, preprint 356,
  http://www.jyu.fi/science/laitokset/maths/tutkimus/preprints.

\bibitem{JarvenpaaJarvenpaaKaenmakiSuomala2005}
E.~J{\"a}rvenp{\"a}{\"a}, M.~J{\"a}rvenp{\"a}{\"a}, A.~K{\"a}enm{\"a}ki, and V.~Suomala.
\newblock Asympotically sharp dimension estimates for {$k$}-porous sets.
\newblock {\em Math. Scand.}, 97(2):309--318, 2005.

\bibitem{KaenmakiSuomala2004}
A.~K{\"a}enm{\"a}ki and V.~Suomala.
\newblock Nonsymmetric conical upper density and {$k$}-porosity.
\newblock {\em Trans. Amer. Math. Soc.}, 2004.
\newblock to appear.

\bibitem{KaenmakiSuomala2008}
A.~K{\"a}enm{\"a}ki and V.~Suomala.
\newblock Conical upper density theorems and porosity of measures.
\newblock {\em Adv. Math.}, 217(3):952--966, 2008.

\bibitem{Marstrand1954}
J.~M.~Marstrand.
\newblock Some fundamental geometrical properties of plane sets of fractional
  dimensions.
\newblock {\em Proc. London Math. Soc. (3)}, 4:257--301, 1954.

\bibitem{Mattila1988}
P.~Mattila.
\newblock Distribution of sets and measures along planes.
\newblock {\em J. London Math. Soc. (2)}, 38(1):125--132, 1988.

\bibitem{Mattila1995}
P.~Mattila.
\newblock {\em Geometry of Sets and Measures in Euclidean Spaces: Fractals and
  Rectifiability}.
\newblock Cambridge University Press, Cambridge, 1995.

\bibitem{MorseRandolph1944}
A.~P.~Morse and J.~F.~Randolph.
\newblock The $\phi$ rectifiable subsets of the plane.
\newblock {\em Trans. Amer. Math. Soc.}, 55:236--305, 1944.

\bibitem{Rajala2009}
T.~Rajala.
\newblock Large porosity and dimension of sets in metric spaces.
\newblock {\em Ann. Acad. Sci. Fenn. Math.}, 2009.
\newblock to appear.

\bibitem{Salli1985}
A.~Salli.
\newblock Upper density properties of {H}ausdorff measures on fractals.
\newblock {\em Ann. Acad. Sci. Fenn. Ser. A I Math. Dissertationes No. 55},
  1985.

\end{thebibliography}

%\newcommand{\etalchar}[1]{$^{#1}$}

\end{document}